\documentclass[11pt]{amsart}

\usepackage{amssymb,amsmath}

\usepackage[T1]{fontenc} 
\usepackage[utf8]{inputenc}  
\usepackage{csquotes}

\allowdisplaybreaks

\usepackage{vmargin}
\setmarginsrb{3cm}{2cm}{3cm}{2cm}{0cm}{2cm}{0cm}{1cm}

\usepackage{enumitem} 
\usepackage{hyperref}
\usepackage{cleveref}

\usepackage{nicematrix}
\usepackage{array}
\newcolumntype{R}{>{\displaystyle}r}
\newcolumntype{L}{>{\displaystyle}l}

\usepackage{xcolor}
\usepackage{todonotes}
\usepackage{soul}

\theoremstyle{plain}
\newtheorem{theorem}{Theorem}
\newtheorem{lemma}[theorem]{Lemma}
\newtheorem{corollary}[theorem]{Corollary}
\newtheorem{proposition}[theorem]{Proposition}

\theoremstyle{definition}
\newtheorem{definition}[theorem]{Definition}
\newtheorem{example}[theorem]{Example}
\newtheorem{remark}[theorem]{Remark} 
\numberwithin{equation}{section}

\AddToHook{cmd/lemma/before}{\crefalias{theorem}{lemma}} 
\AddToHook{cmd/corollary/before}{\crefalias{theorem}{corollary}} 
\AddToHook{cmd/proposition/before}{\crefalias{theorem}{proposition}} 
\AddToHook{cmd/definition/before}{\crefalias{theorem}{definition}} 
\AddToHook{cmd/example/before}{\crefalias{theorem}{example}} 
\AddToHook{cmd/remark/before}{\crefalias{theorem}{remark}} 

\newcommand{\R}{\mathbb R}
\newcommand{\Z}{\mathbb Z}
\newcommand{\N}{\mathbb N}

\newcommand{\val}{\mathrm{val}}
\newcommand{\lex}{\mathrm{lex}}
\DeclareMathOperator{\rad}{rad}

\makeatletter
\newcommand*\bigcdot{\mathpalette\bigcdot@{.5}}
\newcommand*\bigcdot@[2]{\mathbin{\vcenter{\hbox{\scalebox{#2}{$\m@th#1\bullet$}}}}}
\makeatother

\title{Parry condition, existence and uniqueness of alternate bases}
\author{\'Emilie Charlier$^{1,*}$, Savinien  Kreczman$^{1}$, Zuzana Mas\'akov\'a$^2$ and Edita Pelantov\'a$^2$}
\address{$^1$Department of Mathematics\\
University of Li\`ege\\
All\'ee de la D\'ecouverte 12,
4000 Li\`ege, Belgium\\
$^2$Department of Mathematics FNSPE\\
Czech Technical University in Prague\\
Trojanova 13, 120 00 Praha 2, Czech Republic}
 \thanks{\emph{E-mail:} \url{echarlier@uliege.be, savinien.kreczman@uliege.be, zuzana.masakova@fjfi.cvut.cz, edita.pelantova@fjfi.cvut.cz}\\
$^*$Corresponding author.}


\date{\today}

\begin{document}

\begin{abstract}
Alternate bases are a numeration system that generalizes the Rényi numeration system. It is common in this context to construct examples or counter-examples by specifying the expansions of $1$ in the desired system. While it is easy to show when a system with given expansions of $1$ exists in the Rényi case, the same is not true in the alternate case.
In this article, we establish conditions for given words to be the expansions of $1$ in the alternate case. To do so, we use a fixed point theorem on matrices defined from the expansions and obtain the elements of the base from the components of the fixed point.
We also obtain a partial result for the uniqueness of such a base.
In the latter parts of the article, we use similar techniques to prove the existence of bases with a given sequence of $B$-integers.
\end{abstract}

\maketitle

\section{Introduction}

Numeration systems with Cantor real bases are generalizations of Rényi numeration systems. Rényi~\cite{Renyi:1957} considered a real base $\beta>1$ and expressed each non-negative real number $x$ as a series 
\[
    x=\sum_{n\in\Z}a_n\beta^n
\]
with digits $a_n\in\Z\cap[0,\beta)$ with the condition that $a_n=0$ for $n$ greater than some rank. Any sequence of non-negative integers $(a_n)_{n\in \Z}$ with this property is a $\beta$-representation of the number $x$. In general, the $\beta$-representation is not unique, and one chooses the canonical one with a greedy algorithm. This canonical expansion is called the $\beta$-expansion of $x$. Not every sequence of non-negative integers serves as the $\beta$-expansion of some $x$. These `admissible' sequences can be identified by the Parry condition~\cite{Parry:1960} using a lexicographic comparison with the so-called quasi-greedy expansion of 1, which is defined as the lexicographically greatest sequence $d_1d_2d_3\cdots$ of non-negative integers with infinitely many non-zeros such that 
\[
    1=\sum_{n\geq 1}\frac{d_n}{\beta^n}.
\]
This sequence is denoted by $d_\beta^*(1)$. A sequence of non-negative integers is admissible if and only if its suffixes are all lexicographically smaller than  $d_\beta^*(1)$.

It is quite straightforward to show that a sequence $a_1a_2a_3\cdots$ of non-negative integers not ultimately vanishing serves as the quasi-greedy expansion of $1$ for some base $\beta>1$ if and only if its suffixes are all lexicographic less than or equal to the sequence itself. The proof is done in two steps. First, one realizes that there is a unique solution $X=\beta^{-1}\in(0,1)$ of the equation
\[
    1=\sum_{n\ge 1}a_nX^n
\]
by monotonicity of the map $x\mapsto\sum_{n\ge 1}a_nx^n$ on the interval $(0,1)$.
Second, the lexicographic condition ensures that the $\beta$-representation $a_1a_2a_3\cdots$ is lexicographically 
maximal, i.e., it is the quasi-greedy expansion of $1$ for $\beta$.

Cantor real base systems generalize Rényi systems by replacing the sequence of powers of the same base $\beta>1$ by products of members of a bi-infinite sequence of bases. The definition was developed first for numbers $x\in[0,1)$ in~\cite{Charlier&Cisternino:2021} and independently in~\cite{Caalim&Demegillo:2020}, later for any non-negative $x$ in~\cite{Charlier&Cisternino&Masakova&Pelantova:2025}. Consider $B=(\beta_n)_{n\in\Z}$ with $\beta_n>1$. A $B$-representation of a non-negative real number $x$ is a sequence of digits $(a_n)_{n\in \Z}$ such that 
\[
    x=\sum_{n=0}^{+\infty} a_n\beta_{n-1}\cdots\beta_0
    +\sum_{n=1}^{+\infty}\frac{a_{-n}}{\beta_{-1}\cdots\beta_{-n}},
\]
with the condition that $a_n = 0$ for $n$ greater than some rank. In general, this expression is not unique and one chooses the canonical one, called the $B$-expansion, by using the greedy algorithm. The language of $B$-expansions can be described by an analogue of the Parry condition, which lexicographically compares the suffixes of the sequence of digits with the so-called quasi-greedy expansions of 1 in base $B$ and its shifted versions; see \Cref{thm:CantorParryCondition}. It is not difficult to show that the sequence $d_B^*(1)$ itself satisfies such lexicographic conditions with, however, equality allowed; see \Cref{cor:CantorParryCondition}.

Note that if $B$ is a constant sequence $\beta_n=\beta$, the system becomes the Rényi numeration system. A more interesting system, but still less complicated than the general Cantor real base system, is if the base $B$ is periodic. When the minimal period is denoted by $p$, we speak of alternate base of length $p$.

A common way of constructing examples or counter-examples in this setting is to specify not the value of the bases $\beta_n$, but rather those of the quasi-greedy expansions of $1$ with respect to all the shifted bases $S^N(B)=(\beta_{n+N})_{n\in\Z}$. To do so we must ensure that the given words are indeed realizable as expansions of $1$ in some base. As explained above, in the setting of Rényi numeration systems, this problem is well understood. However, even for alternate bases, the problem is harder and only partial results are known -- for instance, \cite[Theorem 6.4]{Charlier&Cisternino&Masakova&Pelantova:2025} obtains uniqueness but only among some alternate bases. The existence of bases with given expansions is often verified by ad hoc computation. The aim of this article is to close this gap in knowledge in the alternate case. 

The first part of our paper is concerned with the following questions. Given a list of $p$ sequences $\mathbf{a}_0,\ldots,\mathbf{a}_{p-1}$, does there exist an alternate base $B$ of length $p$ such that these words are indeed the greedy or quasi-greedy expansions with respect to the $p$ shifts of the base? Is such a base unique?

The answer to this problem is formulated in \Cref{thm:existenceAlternateNonParry}, which is our main result. The proof is divided into several steps. 
First, we show that the lexicographic conditions characterize the greedy and quasi-greedy expansions of $1$ in a general Cantor real base (\Cref{sec:Parry-conditions}). Then, we address the question of the uniqueness of an alternate base with a given list of quasi-greedy expansions of $1$. More precisely, we show that two distinct alternate bases of a given length $p$ cannot share the same list of $p$ quasi-greedy expansions of $1$ with respect to the shifted bases $S^i(B)$, for $i\in\{0,\ldots,p-1\}$,  provided that these sequences all start with a digit greater than $1$ (\Cref{sec:uniqueness}).
Next, we turn to the existence part of the result. We start by deriving some tools from a fixed point theorem of Furstenberg (\Cref{sec:Furstenberg}) concerning products of matrices. This theorem can be seen as an extension of the Perron-Frobenius theorem for irreducible matrices. It associates a generalized eigenvector with an infinite sequence of non-negative matrices. 
Then we show that given a list of $p$ sequences that are lexicographically greater than $10^\omega$, there exists an alternate base $B$ such that these sequences are representations of $1$ with respect to the shifted bases $S^i(B)$, for $i\in\{0,\ldots,p-1\}$ (\Cref{sec:existence}). We proceed by defining a sequence of non-negative matrices from the desired representations of $1$, then use the coefficients of its generalized eigenvectors to construct our alternate base. 
Finally, assuming that the $p$ given sequences satisfy the lexicographic conditions, we derive that they are indeed equal to the greedy or quasi-greedy expansions of $1$ with respect to the $p$ shifted bases. 

In the second part of the paper we use our results in order to show that Cantor real base numeration systems can be used to provide a new interpretation to a large family of infinite words over finite alphabets, including the well-known and extensively studied Arnoux-Rauzy words, or words associated with $N$-continued fractions as introduced by Langeveld, Rossi and Thuswaldner \cite{Langeveld&Rossi&Thuswaldner:2023}. Indeed, these infinite words are found among the symbolic codings of sequences of real numbers with no fractional part with respect to a suitably chosen Cantor real base $B$ (\Cref{sec:specific}).

\section{Preliminaries}
\label{sec:prelim}

In this section, we recall the basic notions associated with Cantor real bases that will be relevant throughout this article. We also set some notation.

\emph{Cantor real bases} \cite{Caalim&Demegillo:2020,Charlier&Cisternino:2021} offer a generalization of Rényi numeration systems. Rather than being directed by a single base, a whole sequence is given as the base of the numeration system. We let $B = (\beta_n)_{n \in \Z}$ be the base sequence and ask that $\beta_n>1$ for all $n$ and $\prod_{n=1}^{+\infty} \beta_n=\prod_{n=1}^{+\infty} \beta_{-n}=+\infty$. We consider bi-infinite words $\mathbf{a}=(a_n)_{n\in\Z}$ over $\N$ with a left tail of zeros, i.e., for which there exists an $N\in\Z$ such that $a_n=0$ for all $n\ge N$ and $a_{N-1}\ne 0$. We write 
\[
	\mathbf{a}=
    \begin{cases}
	a_{N-1}\cdots a_0 \bigcdot a_{-1}a_{-2}\cdots, 	& \text{if }N\ge 1;	\\
	\hspace{1.55cm} 0 \bigcdot 0^{-N}a_{N-1} a_{N-2}\cdots,		& \text{if }N\le 0.
	\end{cases}
\]
We define the evaluation of such words as
\[
    \val_B(\mathbf{a})
    =\sum_{n=0}^{+\infty} a_n\beta_{n-1}\cdots\beta_0 
    + \sum_{n=1}^{+\infty} \frac{a_{-n}}{\beta_{-1}\cdots\beta_{-n}} 
\]
where the first sum is always finite due to the left tail of zeros. If $\val_B(\mathbf{a})<+\infty$, then we say that $\mathbf{a}$ is a \emph{$B$-representation} of the number $\val_B(\mathbf{a})$. 

We select the largest representation in the radix order to be the (canonical) \emph{$B$-expansion} of a given number, where the radix order $<_{\rad}$ is defined by $\mathbf{a}<_{\rad}\mathbf{b}$ if there exists an index $i\in\Z,\ a_j=b_j \text{ for all }j> i$ and $ a_i<b_i$. We denote the $B$-expansion of the number $x$ by $\langle x\rangle_B$. It can be interpreted with a greedy algorithm iteratively computing its digits $a_n$ in decreasing indexing. For $x\ge 1$, let $N$ be the minimal integer such that $x<\beta_{N-1}\cdots\beta_0$ and for $x\in[0,1)$, let $N=0$. Set $a_n=0$ for $n\ge N$. Then set $r_{N-1}=\frac{x}{\beta_{N-1}\cdots\beta_0}$ and for all $n\le N-1$, define $a_n=\lfloor\beta_nr_n\rfloor$ and $r_{n-1}=\beta_nr_n-a_n$. The following lemma gives a way to check that a given sequence is a greedy expansion of some real number, see \cite[Corollary 11]{Charlier&Cisternino:2021} and \cite[Lemma 3.1]{Charlier&Cisternino&Masakova&Pelantova:2025}.

\begin{lemma}
\label{lem:greedy}
Let $B=(\beta_n)_{n\in \Z}$ be a Cantor real base. A sequence $(a_n)_{n\in\Z}$ of non-negative integers, with $a_n=0$ for sufficiently large positive $n$, is the $B$-expansion of some non-negative real number if and only if we have
\[
	\sum_{n=0}^k a_n\beta_{n-1}\cdots\beta_0+\sum_{n=1}^{\infty}\frac{a_{-n}}{\beta_{-1}\cdots\beta_{-n}} < \beta_k\cdots\beta_0\quad\text{for all }k\ge 0
\]
and
\[
	\sum_{n=k}^{\infty}\frac{a_{-n}}{\beta_{-1}\cdots\beta_{-n}} < \frac{1}{\beta_{-1}\cdots\beta_{-k+1}}\quad\text{for all }k\ge 1.
\]
\end{lemma}

In what follows, we will be working not only with the base $B$ itself but also with its shifts. We define $S(B) = (\beta_{n+1})_{n \in \Z}$. Iterating the map $S$, we get $S^i(B) = (\beta_{n+i})_{n \in \Z}$ for $i\in\Z$. If $S^p(B)=B$ for some $p\ge 1$, meaning that the sequence $B$ is periodic with period $p$, we speak of an \emph{alternate base}. In this case, we will abuse notation by simply writing $B=(\beta_{p-1},\ldots,\beta_0)$, and we call $p$ the \emph{length} of $B$. 

In our setting, the $B$-expansion of $1$ is always given by $\langle 1\rangle=1\bigcdot 0^\omega$. Two other $B$-representations of $1$ will be of particular interest to us. For every $n\in\Z$, we define two (one-sided) infinite words $\mathbf{t}_n$ and $\mathbf{d}_n$ as follows. The infinite word $\mathbf{t}_n$ is the lexicographically greatest infinite word $\mathbf{t}$ such that $\val_{S^n(B)}(0\bigcdot \mathbf{t})=1$, whereas $\mathbf{d}_n$ is defined as the lexicographically greatest infinite word $\mathbf{d}$ not ending in a tail of zeros such that $\val_{S^n(B)}(0\bigcdot \mathbf{d})=1$. Clearly, the words $\mathbf{t}_n$ and $\mathbf{d}_n$ may differ only when $\mathbf{t}_n$ ends in a tail of zeros. By a slight abuse of language, we will talk about $\mathbf{t}_n$ and $\mathbf{d}_n$ as the \emph{greedy} and \emph{quasi-greedy} $S^n(B)$-expansions of $1$, respectively. We let $t_{n,i}$ and $d_{n,i}$, with $ i\ge 1$, denote the digits of these words:
\[
    \mathbf{t}_n =t_{n,1}t_{n,2}t_{n,3}\cdots
    \quad\text{ and }\quad
    \mathbf{d}_n = d_{n,1}d_{n,2}d_{n,3}\cdots
\]
Note that we use an increasing indexing for these (one-sided) infinite words. The quasi-greedy expansions of $1$ can be recursively obtained from the greedy expansions of $1$ as follows:
\[
\mathbf{d}_n = 
\begin{cases}
    t_{n,1}\cdots t_{n,\ell-1}(t_{n,\ell}-1)\mathbf{d}_{n-\ell}, & \text{ if } \mathbf{t}_n=t_{n,1}\cdots t_{n,\ell}0^\omega \text{ with }t_{n,\ell}\ge 1,\\
    \mathbf{t}_n, & \text{ otherwise}. 
\end{cases}
\]
Equivalently, the quasi-greedy expansions of $1$ are obtained as limits with respect to the product topology on infinite words: for all $n\in\Z$, we have $\lim_{x\to 1^-} \langle x\rangle_{S^n(B)}=0\bigcdot \mathbf{d}_n$.

The interest of greedy and quasi-greedy expansions of $1$ lies in the fact that they can be used to determine when a given $B$-representation of a number is its $B$-expansion. See for instance the following result.

\begin{theorem}[{\cite[Corollary 4.4]{Charlier&Cisternino&Masakova&Pelantova:2025}}]
\label{thm:CantorParryCondition}
A $B$-representation $a_{N-1}\cdots a_0 \bigcdot a_{-1}a_{-2}\cdots$  of a non-negative real number is its $B$-expansion if and only if $a_{n-1}a_{n-2}\cdots <_{\lex} \mathbf{d}_n$ for all $n\le N$.
\end{theorem}

Focusing on representations of $1$, we obtain the following corollary.

\begin{corollary}
\label{cor:CantorParryCondition}  
Let $B=(\beta_n)_{n\in\Z}$ be a Cantor real base. For all $n\in\Z$, we have the following properties.
\begin{enumerate} 
\item The sequences $\mathbf{t}_n$ and $\mathbf{d}_n$ start with a non-zero digit.
\item The sequence $\mathbf{d}_n$ does not end in a tail of zeros.
\item We have $\val_{S^n(B)}(0\bigcdot \mathbf{t}_n)=\val_{S^n(B)}(0\bigcdot \mathbf{d}_n)=1$.
\item A sequence $\mathbf{a}_n=a_{n,1}a_{n,2}\cdots$ of non-negative integers such that $\val_{S^n(B)}(\mathbf{a}_n)=1$ is equal to $\mathbf{t}_n$ if and only if 
\[
    \val_{S^{n-j}(B)} (0\bigcdot a_{n,j+1}a_{n,j+2}\cdots) <1
\]
for all $j\ge 1$.
\item (The so-called Parry conditions) We have 
\begin{equation}
\label{eq:Parry-t}
    t_{n,j+1}t_{n,j+2}\cdots <_{\lex} \mathbf{d}_{n-j}
\end{equation}
and 
\begin{equation}
\label{eq:Parry-d}
    d_{n,j+1}d_{n,j+2}\cdots \le_{\lex} \mathbf{d}_{n-j}.
\end{equation}
for all $j\ge 1$.
\end{enumerate}
\end{corollary}

\begin{proof}
 The first four properties were shown in \cite{Charlier&Cisternino:2021}. The Parry condition \eqref{eq:Parry-t} for the greedy expansions $\mathbf{t}_n$ is a direct consequence of \Cref{thm:CantorParryCondition}. 
 Let now $n\in\Z$ be fixed and let us show the Parry condition for the quasi-greedy expansion $\mathbf{d}_n$. 
 Let $(\ell_k)_{k\ge 0}$ be the (finite or infinite) sequence of lengths of the finite greedy-expansions (i.e., not ending in a tail of zeros) that are encountered in the computation of $\mathbf{d}_n$. If $j=\ell_1+\cdots+\ell_k$ for some $k$, then $d_{n,j+1}d_{n,j+2}\cdots=\mathbf{d}_{n-j}$, in which case we are fine. Otherwise, choose $k$ maximal such that $\ell:=\ell_1+\cdots+\ell_k<j$. Then $d_{n,j+1}d_{n,j+2}\cdots \le_{\lex} t_{n-\ell,j-\ell+1}t_{n-\ell,j-\ell+2}\cdots$, hence we obtain that $d_{n,j+1}d_{n,j+2}\cdots <_{\lex} \mathbf{d}_{n-j}$ by using the Parry condition for $\mathbf{t}_{n-\ell}$. 
\end{proof}

\section{The Parry conditions characterize the greedy and the quasi-greedy expansions of 1}
\label{sec:Parry-conditions}

First, we recall the following result from \cite{Charlier&Cisternino:2021}. 

\begin{lemma}[{\cite[Lemma 25]{Charlier&Cisternino:2021}}]
\label{lem:Parry-CC}
Let $B$ be a Cantor real base and for all $n\le 0$, let $\mathbf{a}_n=a_{n,1}a_{n,2}\cdots$ be a sequence of non-negative integers such that $\val_{S^n(B)}(0\bigcdot \mathbf{a}_n)=1$. If $\mathbf{b}=b_1b_2\cdots$ is a sequence of non-negative integers such that $b_{n+1}b_{n+2}\cdots<_{\lex}\mathbf{a}_{-n}$ for all $n\ge 0$, then we have $\val_{S^{-n}(B)}(0\bigcdot b_{n+1}b_{n+2}\cdots)<1$ for all $n\ge 0$, unless there exist $n\ge 0$ and $\ell\ge 1$ such that
\begin{itemize}
    \item $\mathbf{a}_{-n}=a_{-n,1}\cdots a_{-n,\ell}0^\omega$ where $a_{-n,\ell}\ne 0$,
    \item $b_{n+1}\cdots b_{n+\ell}=a_{-n,1}\cdots a_{-n,\ell-1}(a_{-n,\ell}-1)$,
    \item $\val_{S^{-n-\ell}(B)}(b_{n+\ell+1}b_{n+\ell+2}\cdots)=1$,
\end{itemize}
in which case we have $\val_{S^{-n}(B)}(0\bigcdot b_{n+1}b_{n+2}\cdots)=1$.
\end{lemma}

With the next result, we show that the Parry conditions from \Cref{cor:CantorParryCondition} characterize the greedy and the quasi-greedy expansions of $1$.

\begin{proposition}
\label{prop:uniqueness}
Let $B$ be a Cantor real base and for all $n\in \Z$, let $\mathbf{a}_n=a_{n,1}a_{n,2}\cdots$ be a sequence of non-negative integers such that $\val_{S^n(B)}(0\bigcdot \mathbf{a}_n)=1$. Assume moreover that for all $n\in\Z$, if $\mathbf{a}_n$ ends in $0^\omega$ then
\[
    a_{n,j+1}a_{n,j+2}\cdots <_{\lex} \mathbf{a}_{n-j} \text{ for all }j\ge 1,
\]
and if $\mathbf{a}_n$ does not end in $0^\omega$ then
\[
    a_{n,j+1}a_{n,j+2}\cdots \le_{\lex} \mathbf{a}_{n-j} \text{ for all }j\ge 1.
\]
Then we have $\mathbf{a}_n=\mathbf{t}_n$ for all $n\in\Z$ such that $\mathbf{a}_n$ ends in a tail of zeros, and we have $\mathbf{a}_n=\mathbf{d}_n$ for all $n\in\Z$ such that $\mathbf{a}_n$ does not end in a tail of zeros.
\end{proposition}

\begin{proof}
Let $n\in\Z$ be fixed. First suppose that $\mathbf{a}_n$ ends in a tail of zeros. By \Cref{cor:CantorParryCondition}, in order to get that $\mathbf{a}_n=\mathbf{t}_n$, it suffices to show that $\val_{S^{n-j}(B)} (0\bigcdot a_{n,j+1}a_{n,j+2}\cdots) <1$ for all $j\ge 1$. Assume by contradiction that there exists some $j\ge 1$ with $\val_{S^{n-j}(B)} (0\bigcdot a_{n,j+1}a_{n,j+2}\cdots) \ge 1$. Since $\mathbf{a}_n$ ends in $0^\omega$, there exist only finitely many such $j$. We pick the greatest such $j$. Then \Cref{lem:Parry-CC} implies that there exists some $\ell>j$ such that $\val_{S^{n-\ell}(B)}(0\bigcdot a_{n,\ell+1}a_{n,\ell+2}\cdots)=1$, which obviously contradicts the maximality of $j$. 

Second, suppose that $\mathbf{a}_n$ does not end in a tail of zeros. Let $(x_\ell)_{\ell\ge 0}$ be the sequence of numbers defined by $x_\ell=\val_{S^n(B)}(0\bigcdot a_{n,1}\cdots a_{n,\ell}0^\omega)$. We have $\lim_{\ell\to+\infty}x_\ell=1$. Let now fix an $\ell\ge 0$. By \Cref{lem:Parry-CC}, we know that $\val_{S^{n-j}(B)}(0\bigcdot a_{n,j+1}\cdots a_{n,\ell}0^\omega)<1$ for all $j\in\{0,\ldots,\ell-1\}$. Then \Cref{lem:greedy} tells us that $0\bigcdot a_{n,1}\cdots a_{n,\ell}0^\omega=\langle x_\ell\rangle_{S^n(B)}$. We obtain that $\mathbf{a}_n=\mathbf{d}_n$ by letting $\ell$ tend to infinity.
 \end{proof}

The following two corollaries are immediate.

\begin{corollary}
 \label{cor:greedyExpansion}
 Let $B$ be a Cantor real base and for all $n\in \Z$, let $\mathbf{a}_n=a_{n,1}a_{n,2}\cdots$ be a sequence of non-negative integers ending in a tail of zeros such that $\val_{S^n(B)}(0\bigcdot \mathbf{a}_n)=1$ and $a_{n,j+1}a_{n,j+2}\cdots <_{\lex} \mathbf{a}_{n-j}$ for all $j\ge 1$. Then $\mathbf{a}_n=\mathbf{t}_n$ for all $n\in\Z$.
 \end{corollary}

\begin{corollary}
 \label{cor:quasi-greedyExpansion}
Let $B$ be a Cantor real base and for all $n\in \Z$, let $\mathbf{a}_n=a_{n,1}a_{n,2}\cdots$ be a sequence of non-negative integers not ending in a tail of zeros such that $\val_{S^n(B)}(0\bigcdot \mathbf{a}_n)=1$ and $a_{n,j+1}a_{n,j+2}\cdots \le_{\lex} \mathbf{a}_{n-j}$ for all $j\ge 1$. Then $\mathbf{a}_n=\mathbf{d}_n$ for all $n\in\Z$.
 \end{corollary}

\section{Uniqueness of the alternate base with prescribed greedy or quasi-greedy expansions of 1}
\label{sec:uniqueness}

In \cite{Charlier&Cisternino&Masakova&Pelantova:2025}, it was shown that there is at most one alternate base with a given list of ultimately periodic quasi-greedy expansions of $1$.
In this section, we partially solve the question on the uniqueness of the base $B$ in the non-Parry case. 

\begin{proposition}
\label{pro:uniqueness}
Let $p\ge 1$ be an integer, let $\alpha$ be the zero of the polynomial $X^p-X^{p-1}-\cdots-1$ in $[1,2)$, and let $B=(\beta_{p-1},\ldots,\beta_0)$ be an alternate base such that $\beta_i>\alpha$ for every~$i$. Then no other alternate base of length $p$ has the same list of greedy (resp.\ quasi-greedy) expansions of $1$. 
\end{proposition}

\begin{proof} 
Let $\Gamma=(\gamma_{p-1},\ldots,\gamma_0)$ be an alternate base such that for all $i\in\{0,\ldots,p-1\}$, the greedy $S^i(\Gamma)$-expansion of $1$ coincides with the greedy $S^i(B)$-expansion of $1$, which we denote $\mathbf{t}_i=t_{i,1}t_{i,2}t_{i,3}\cdots$. 

Let $i\in\{0,\ldots,p-1\}$. The digits $t_{i,n}$ are obtained thanks to the greedy algorithm as follows. We set $r_{i,0}=1$ and for all $n\ge 1$, set $t_{i,n}=\lfloor\gamma_{i-n}r_{i,n-1}\rfloor$ and $r_{i,n}=\gamma_{i-n}r_{i,n-1}-t_{i,n}$. Thus, we have $r_{i,n}\in[0,1)$ and $t_{i,n}=\gamma_{i,n}r_{i,n-1}-r_{i,n}$ for all $n\ge 1$. We then have
\[
    1   = \val_{S^i(B)}(0\bigcdot \mathbf{t}_i) 
        = \sum_{n=1}^{+\infty}\frac{t_{i,n}}{\beta_{i-1}\cdots\beta_{i-n}} 
        = \sum_{n=1}^{+\infty}\frac{\gamma_{i-n}r_{i,n-1}-r_{i,n}}{\beta_{i-1}\cdots\beta_{i-n}}.
\]
By using that $r_{i,0}=1$ and by rearranging the terms, we obtain
\[
    0 = \sum_{n=0}^{+\infty}\left(\frac{\gamma_{i-n-1}}{\beta_{i-n-1}}-1\right)\frac{r_{i,n}}{\beta_{i-1}\cdots\beta_{i-n}}.
\]
Grouping the terms $p$ by $p$ and by setting $\delta=\beta_0\cdots\beta_{p-1}$, we then get
\[
    0 = \sum_{j=0}^{p-1}\left(\frac{\gamma_{i-j-1}}{\beta_{i-j-1}}-1\right)\frac{1}{\beta_{i-1}\cdots\beta_{i-j}}\sum_{n=0}^{+\infty}\frac{r_{i,pn+j}}{\delta^n}.
\]
For all $j\in\{0,\ldots,p-1\}$, by setting 
\[
    E_{i,j}=\frac{1}{\beta_{i-1}\cdots\beta_{i-j}}\sum_{n=0}^{+\infty}\frac{r_{i,pn+j}}{\delta^n}
\]
and 
\[
    F_{i,j}=\frac{\gamma_{i-j-1}}{\beta_{i-j-1}}-1,
\]
we can rewrite the latter equality as
\[
    0=\sum_{j=0}^{p-1} E_{i,j} F_{i,j}.
\]

Observe that for $i\in\{0,\ldots,p-1\}$, the row matrices 
\[
    F_i=\begin{pmatrix}F_{i,0} & \cdots &F_{i,p-1} \end{pmatrix}
\]
are circular permutations of one another. In particular, we have 
\[
    F_0=\begin{pmatrix}F_{i,i} & \cdots &F_{i,i+p-1} \end{pmatrix}
\]
where the second indices are seen modulo $p$. We obtain the matrix equality 
\[
    EF_0^\intercal=0
\]
where 
\[
    E=\begin{pmatrix}
        E_{0,0}     & E_{0,1}   & \cdots    & E_{0,p-2}     & E_{0,p-1} \\ 
        E_{p-1,p-1} & E_{p-1,0} & \cdots    & E_{p-1,p-3}   & E_{p-1,p-2} \\
        \vdots      & & & & \\
        E_{1,1}     & E_{1,2}   & \cdots    & E_{1,p-1}     & E_{1,0}
    \end{pmatrix}.
\]
The matrix $E$ has non-negative entries. Let us show that it is diagonally dominant, i.e., that we have 
\begin{equation}
\label{eq:diagonally-dominant}
    E_{i,0}>\sum_{j=1}^{p-1} E_{i,j}
\end{equation}
for all $i\in\{0,\ldots,p-1\}$. 
Indeed, we have $E_{i,0}\geq r_{i,0}=1$ for all $i\in\{0,\ldots,p-1\}$. If $p=1$, this is already enough.
Otherwise, by using the hypothesis that $\beta_i>\alpha$ for all $i\in\{0,\ldots,p-1\}$ and by using that $r_{i,n}\le 1$ for all $n\ge 0$, we obtain that 
\[
    \sum_{j=1}^{p-1}E_{i,j}
    <\sum_{j=1}^{p-1}\frac{1}{\alpha^j}\frac{\alpha^p}{\alpha^p-1}
    =\frac{\sum_{j=1}^{p-1}\alpha^{p-j}}{\alpha^p-1}
    =1
    \leq E_{i,0}
\]
as announced. 

Hadamard's lemma then tells us that $\det(E)\ne 0$, which implies in turn that $F_0=0$. We then get that $\gamma_i=\beta_i$ for all $i\in\{0,\ldots,p-1\}$, that is, $\Gamma=B$.

The proof for the quasi-greedy expansions is similar. The digits of the quasi-greedy expansions $\mathbf{d}_n=d_{n,1}d_{n,2}\cdots$ are obtained by the following quasi-greedy algorithm. We set $r_{i,0}=1$ and for all $n\ge 1$, set $d_{i,n}=\lceil{\gamma_{i-n}r_{i,n-1}}\rceil-1$ and $r_{i,n}=\gamma_{i-n}r_{i,n-1}-d_{i,n}$. Thus, we have $r_{i,n}\in(0,1]$ and $d_{i,n} = \gamma_{i-n}r_{i,n-1}-r_{i,n}$ for all $n\ge 1$. The rest of the proof is easily adapted.
\end{proof}

\begin{corollary} 
\label{cor:uniqueness2}
Let $p\ge 1$ be an integer and let $B$ be an alternate base of length $p$ such that for all $i\in \{0,1,\ldots, p-1\}$, the greedy (resp.\ the quasi-greedy) $S^i(B)$-expansion of~$1$ starts with a digit that is at least 2. Then no other alternate base of length~$p$ has the same list of greedy (resp.\ quasi-greedy) expansions of $1$. 
\end{corollary}

\section{Consequences of a theorem of Furstenberg}
\label{sec:Furstenberg}

In this section, we prepare a tool that we will use to prove the existence of Cantor real bases with prescribed representations of $1$. The origin of the following theorem can be traced back to Furstenberg \cite{Furstenberg:1960}; also see \cite{Birkhoff:1957}. We present this result in a matrix formulation similar to the one often used when dealing with $S$-adic sequences \cite{Berthe&Tijdeman:2004,Berthe&Delecroix:2014,Berthe&Minervino&Steiner&Thuswaldner:2016}. We note, however, that the formulation in these articles relies on previous assumptions which ensure that certain quantities are positive, notably the row matrix $f$ of the next statement. As we do not have such hypotheses, we must work in general with nonnegative quantities. When required, we will argue that the quantities at hand are positive. See \Cref{rk:on-nonnegativity-vs-positivity} for an example where $f$ is only nonnegative.

\begin{theorem}[{\cite{Furstenberg:1960}}]
\label{thm:Furstenberg} 
Let $(A_n)_{n\in\N}$ be a sequence matrices in $\mathbb{N}^{k\times k}$. Assume that there exist a positive matrix $P$ and indices $i_1 < j_1 \leq i_2 < j_2\leq  \cdots $ such that 
\[
    \cdots=A_{j_2-1}\cdots A_{i_2}=A_{j_1-1}\cdots A_{i_1}  = P.
\]
Then there exists a row matrix $f\in\R_{\geq 0}^k$ such that
\[
    \bigcap_{n\ge 0} \R_{\ge 0}^k A_{n-1}\cdots A_0 = \R_{\ge 0} f.
\]
\end{theorem}

\begin{proposition}
\label{pro:ConseqFurstenberg} 
Let $(A_n)_{n\in \Z}$ be a sequence of non-negative integer square matrices such that $(A_n)_{1,1}\ge 1$ for all $n\in\Z$ and such that $(A_n)_{n\in \N}$ satisfies the hypotheses of \Cref{thm:Furstenberg}. Then there exist unique sequences $(\gamma_n)_{n\in\Z}$ of positive real numbers and $(f_n)_{n\in\Z}$ of non-negative real row matrices with first entry equal to $1$ such that 
\begin{equation}
\label{eq:recurrence}
    \gamma_n f_{n-1} = f_n A_n
\end{equation}
for all $n\in\Z$. Moreover, if the sequence $(A_n)_{n\in \Z}$ is periodic, then the sequences $(\gamma_n)_{n\in\Z}$ and $(f_n)_{n\in\Z}$ are periodic with the same period. 
\end{proposition}

\begin{proof}
For all $n\in\Z$, since the sequence $(A_{n+1+i})_{i\in\N}$ also satisfies the assumptions of \Cref{thm:Furstenberg}, there exists a non-negative real row matrix $f_n$ such that  
\begin{equation}
\label{eq:f_n}
    \bigcap_{i\ge 0} \R_{\ge 0}^k A_{n+i}\cdots A_{n+1}=\R_{\ge 0} f_n.
\end{equation}
Write $f_n=\begin{pmatrix}
    f_{n,1} & \cdots & f_{n,k}
\end{pmatrix}$.
We begin by proving that the row matrices $f_{i_r-1}$ are positive for all $r\ge 1$, where $(i_r)_{r\ge 1}$ is the sequence of indices from \Cref{thm:Furstenberg}. Indeed, we have $\bigcap_{i\ge 0} \R_{\ge 0}^k A_{i_r+i}\cdots A_{i_r}= \bigcap_{i\ge 0} \R_{\ge 0}^k A_{j_r+i}\cdots A_{j_r} P$, which is an intersection of a decreasing sequence of sets. Since $(A_n)_{1,1}\ge 1$ for all $n$, each product contains a row matrix with positive first entry.
If we consider only row matrices of norm $1$ in each product, we obtain a decreasing sequence of nonempty compact sets, which must have a nonempty intersection. Thus the intersection contains a nonzero element, which in particular belongs to $\R_{\ge 0}^k P$. But this space only contains the zero row matrix and row matrices in $\R_{>0}^k$. Thus $f_{i_r-1}$ is positive, as claimed.

Multiplying each side of \eqref{eq:f_n} by $A_n$ leads to $\R_{\ge 0} f_{n-1} = \R_{\ge 0} f_nA_n$. Now, since $(A_n)_{1,1}\ge 1$, we have that $f_{n,1}>0$ implies $f_{n-1,1}>0$. We deduce that $f_{n,1}>0$ for all $n\in\Z$. Without loss of generality, we may assume that this first entry of $f_n$ is $1$ for all $n\in\Z$.
Since $f_nA_n\in \R_{\geq 0}f_{n-1}$, we find that there exists $\gamma_n\geq 0$ such that $\gamma_n f_{n-1} = f_n A_n$. Given that the first entries of $f_n$ and $f_{n-1}$ are positive, the numbers $\gamma_n$ must be positive.
\end{proof}

\begin{proposition}
\label{pro:EventPeriodic}  
Let $(A_n)_{n\in \Z}$ be a sequence of matrices in $\N^{k\times k}$, with $k\ge 2$, such that $(A_n)_{n\in \N}$ satisfies the assumptions of \Cref{pro:ConseqFurstenberg}. Suppose moreover that there exists $h\in\{1,\ldots,k-1\}$ such that all matrices $A_n$ are of the form
\[
\begin{pNiceArray}{ccc|c}
  a_{n,1} & \cdots & a_{n,k-1} & a_{n,k} \\
  \hline
  \Block{3-3}<\Large>{\mathbf{I}} & \\
  &&& \mathbf{e}_h \\
  &&&
\end{pNiceArray}
\]
where $a_{n,1}\ge 1$, $\mathbf{I}$ is the identity matrix of size $k-1$ and $\mathbf{e}_h$ is the column matrix of size $k-1$ having $1$ in position $h$ and $0$ elsewhere. Otherwise stated, we have
\[
    (A_n)_{i,j}=\begin{cases}
    a_{n,j}, & \text{if }i=1; \\
    1,       & \text{if }i=j+1; \\
    1,       & \text{if }i=h+1 \text{ and }j=k;\\
    0,       & \text{otherwise}
    \end{cases}
\]
with the condition $a_{n,1}\ge 1$. Let $(\gamma_n)_{n\in\Z}$ and $(f_n)_{n\in\Z}$ be the sequences from \Cref{pro:ConseqFurstenberg}, and write $f_n=\begin{pmatrix}
    f_{n,1} & \cdots & f_{n,k}
\end{pmatrix}$.
Then for all $n\in\Z$, we have
\begin{enumerate}
    \item \label{item1} $\gamma_n>1$
    \item \label{item2} $1=\displaystyle{\sum_{j=1}^h} \frac{a_{n+j,j}}{\gamma_{n+1}\cdots\gamma_{n+j}}+\frac{f_{n+h,h+1}}{\gamma_{n+1}\cdots\gamma_{n+h}}$
    \item \label{item3} $f_{n+h,h+1}=\displaystyle{\sum_{j=h+1}^k} \frac{a_{n+j,j}}{\gamma_{n+h+1}\cdots\gamma_{n+j}}+\frac{f_{n+k,h+1}}{\gamma_{n+h+1}\cdots\gamma_{n+k}}$.
\end{enumerate}
\end{proposition}   

\begin{proof} 
We first prove that the row matrices $f_n$ are positive for all $n\in\Z$. As in the proof of \Cref{pro:ConseqFurstenberg}, if $(i_r)_{r\geq 1}$ is the sequence of indices from \Cref{thm:Furstenberg}, then the row matrices $f_{i_r-1}$ are all positive. Since the matrices $A_n$ have a positive entry in each column, any product $f A_n$ where $f$ is a positive row matrix is itself a positive row matrix. Since $\gamma_n f_{n-1}= f_nA_n$, we see that $f_n$ being positive implies that $f_{n-1}$ also is. So we can conclude that $f_n>0$ for all $n$.

Consider now $j \in \{1,\ldots,k\}$. 
Looking at the $j^{\rm th}$ component of \eqref{eq:recurrence} and realizing that $f_{n,1}=1$, we find that
\[
    \gamma_nf_{n-1,j} = 
    \begin{cases}
        a_{n,j} + f_{n,j+1},  & \text{if } j<k;\\
        a_{n,k} + f_{n,h+1},    & \text{if } j=k
    \end{cases}
\]
for all $n\in\Z$. In particular, for $j=1$, we get that $\gamma_n=a_{n,1}+f_{n,2} >1$, which is \Cref{item1}. Let $n\in\Z$. The previous equality taken in $n+j$ instead of $n$ yields 
\begin{equation}
\label{eq:BetaII}  
    \gamma_{n+j}f_{n+j-1,j} = 
    \begin{cases}
        a_{n+j,j} + f_{n+j,j+1},  & \text{if } j<k;\\
        a_{n+j,k} + f_{n+j,h+1},    & \text{if } j=k.
    \end{cases}
\end{equation} 
For $j=1,\ldots,h$, we divide the equality \eqref{eq:BetaII} by $\gamma_{n+1}\cdots \gamma_{n+j}$ and sum up. We obtain
\[
    \sum_{j=1}^h\frac{f_{n+j-1,j}}{\gamma_{n+1}\cdots \gamma_{n+j-1}}
    =\sum_{j=1}^h\frac{a_{n+j,j}}{\gamma_{n+1}\cdots \gamma_{n+j}}
    +\sum_{j=1}^h\frac{f_{n+j,j+1}}{\gamma_{n+1}\cdots \gamma_{n+j}},
\]
which gives
\[
    1+\sum_{j=2}^h\frac{f_{n+j-1,j}}{\gamma_{n+1}\cdots \gamma_{n+j-1}}
    =\sum_{j=1}^h\frac{a_{n+j,j}}{\gamma_{n+1}\cdots \gamma_{n+j}}
    +\sum_{j=2}^{h+1}\frac{f_{n+j-1,j}}{\gamma_{n+1}\cdots \gamma_{n+j-1}}.
\]
Deleting identical terms, we obtain \Cref{item2}. For $j=h+1,\ldots,k$, we divide the equality \eqref{eq:BetaII} by $\gamma_{n+h+1}\cdots \gamma_{n+j}$ and sum up. We obtain 
\[
    \sum_{j=h+1}^k\frac{f_{n+j-1,j}}{\gamma_{n+h+1}\cdots \gamma_{n+j-1}}
    =\sum_{j=h+1}^k\frac{a_{n+j,j}}{\gamma_{n+h+1}\cdots \gamma_{n+j}}
    +\sum_{j=h+1}^{k-1}\frac{f_{n+j,j+1}}{\gamma_{n+h+1}\cdots \gamma_{n+j}}
    +\frac{f_{n+k,h+1}}{\gamma_{n+h+1}\cdots \gamma_{n+k}}.
\]
Simplifying as in the previous case, we obtain \Cref{item3}.
\end{proof}

The following result will be used in \Cref{sec:specific} in order to show that all $S$-adic sequences obtained by using some well-identified family of substitutions can be obtained as the faithful coding of the $B$-integers for some Cantor real base $B$.

\begin{proposition}
\label{pro:FiniteExpansion} 
Let $(A_n)_{n\in \Z}$ be a sequence of matrices in $\N^{k\times k}$, with $k\ge 2$, such that $(A_n)_{n\in \N}$ satisfies the assumptions of \Cref{pro:ConseqFurstenberg}. Suppose moreover that the matrices $A_n$ are of the form
\[
    \begin{pNiceArray}{ccc|c}
    a_{n,1} & \cdots & a_{n,k-1} & a_{n,k} \\
    \hline
    \Block{3-3}<\Large>{\mathbf{I}} & \\
    &&& \mathbf{0} \\
    &&&
    \end{pNiceArray}
\]
where $a_{n,1}\ge 1$, $\mathbf{I}$ is the identity matrix of size $k-1$ and $\mathbf{0}$ is the zero column matrix of size $k-1$. Otherwise stated, we have
\[
    (A_n)_{i,j}=\begin{cases}
    a_{n,j}, & \text{if }i=1; \\
    1,       & \text{if }i=j+1; \\
    0,       & \text{otherwise}
    \end{cases}
\]
with the condition $a_{n,1}\ge 1$. Let $(\gamma_n)_{n\in\Z}$ be the sequence from \Cref{pro:ConseqFurstenberg}.
Then for all $n\in\Z$, we have
\begin{enumerate}
    \item \label{item1:finiteExpansion} $\gamma_n\ge 1$
    \item \label{item2:finiteExpansion} $1=\displaystyle{\sum_{j=1}^k} \frac{a_{n+j,j}}{\gamma_{n+1}\cdots\gamma_{n+j}}$.
\end{enumerate}
Assuming moreover that $a_{n,k}\ge 1$ for all $n\in\Z$, then we have $\gamma_n>1$ for all $n\in\Z$.
\end{proposition}

\begin{proof} 
Let $j\in \{1,\ldots,k\}$. Looking at the $j^{\rm th}$ component of the equality \eqref{eq:recurrence} and realizing that $f_{n,1}=1$, we see that
\[   
    \gamma_n f_{n-1,j}= 
    \begin{cases}
        a_{n,j} + f_{n,j+1},  & \text{if } j<k; \\
        a_{n,k},              & \text{if } j=k.
    \end{cases}
\]
for all $n\in\Z$. In particular, for $j=1$, we find $\gamma_n=a_{n,1}+f_{n,2}\ge 1$, which is \Cref{item1:finiteExpansion}. The latter equality in $n+j$ yields
\begin{equation}
\label{eq:BetaI}
    \gamma_{n+j} f_{n+j-1,j}= 
    \begin{cases}
        a_{n+j,j} + f_{n+j,j+1},  & \text{if } j<k; \\
        a_{n+k,k},                & \text{if } j=k.
    \end{cases}
\end{equation} 
For $j=1,\ldots,k$, dividing \Cref{eq:BetaI} by $\gamma_{n+1}\cdots \gamma_{n+j}$ and summing up, we find
\[
    \sum_{j=1}^k \frac{f_{n+j-1,j}}{\gamma_{n+1}\cdots \gamma_{n+j-1}}
    =\sum_{j=1}^k \frac{a_{n+j,j}}{\gamma_{n+1}\cdots \gamma_{n+j}}
    +\sum_{j=1}^{k-1} \frac{f_{n+j,j+1}}{\gamma_{n+1}\cdots \gamma_{n+j}}.
\]
Deleting identical terms then gives \Cref{item2:finiteExpansion}.

If moreover, we have $a_{n,k}\ge 1$ for all $n\in\Z$, then the matrices $A_n$ all have a positive entry in each column. Arguing as in the beginning of the proof of \Cref{pro:EventPeriodic}, we obtain that $f_n>0$ for all $n$. Then, going back to $\gamma_n=a_{n,1}+f_{n,2}$, we can now deduce that $\gamma_n>1$ for all $n$.
\end{proof}

\begin{remark}\label{rk:on-nonnegativity-vs-positivity}
It is possible to have $\gamma_n=1$ if the condition $a_{n,k}\ge 1$ is not satisfied for all $n$. Consider the matrices
\[
A_2=
\begin{pmatrix}
   1 & 1 & 0\\
   1 & 0 & 0\\
   0 & 1 & 0\\
\end{pmatrix}
,\quad
A_1=
\begin{pmatrix}
   1 & 0 & 1\\
   1 & 0 & 0\\
   0 & 1 & 0\\
\end{pmatrix}
,\quad
A_0=
\begin{pmatrix}
   1 & 1 & 1\\
   1 & 0 & 0\\
   0 & 1 & 0\\
\end{pmatrix},
\]
and define $A_n=A_{n\bmod 3}$ for all $n\in\Z$. This sequence of matrices satisfies the assumptions of \Cref{thm:Furstenberg} since $A_2A_1A_0$ is positive. Nevertheless, using \Cref{item2:finiteExpansion} for $n=-1$, we find that 
\[
1=\frac{a_{0,1}}{\gamma_0}+\frac{a_{1,2}}{\gamma_0\gamma_1}+\frac{a_{2,3}}{\gamma_0\gamma_1\gamma_2},
\]
so $1=\frac{1}{\gamma_0}$ and $\gamma_0=1$. We can find $\gamma_2=\frac{1+\sqrt{17}}{2}$ and $\gamma_1=\frac{3+\sqrt{17}}{4}$ with similar methods. From there, the equalities \eqref{eq:BetaI} allow us to compute for instance 
\[
    f_{0}=\begin{pmatrix}
        1 & 0 &\frac{-3+\sqrt{17}}{2}
        \end{pmatrix}.
\]
\end{remark}

\section{Existence of an alternate base with given representations of 1}
\label{sec:existence}

Given $p$ sequences of digits, we aim to show that there exists an alternate base of length $p$ for which these sequences evaluate to $1$. We start with a lemma in order to be able to make use of Furstenberg's result.

\begin{lemma}
\label{lem:ExistenceParryFurstenbergOK}
Let $p\ge 1$ and for each $i\in\{0,\ldots,p-1\}$, let $\mathbf{a}_i=a_{i,1}a_{i,2}\cdots$ be an ultimately periodic sequence of non-negative integer digits, not starting in $0$ and not ending in $0^\omega$. Let $M,N$ be such that the $p$ sequences $\mathbf{a}_0,\ldots,\mathbf{a}_{p-1}$ all have preperiod $Mp$ and period $Np$. Then the sequence $(A_n)_{n\in\Z}$ of square matrices of size $(M+N)p$ defined by
\[
    (A_n)_{i,j}=
    \begin{cases}
        a_{j-n,j},     & \text{if }i=1;\\
        1,               & \text{if }i=j+1;\\
        1,               & \text{if }i=Mp+1 \text{ and } j=(M+N)p;\\
        0,               & \text{otherwise},
    \end{cases}
\]
where the first index in the digits $a_{m,j}$ is considered modulo $p$, is such that $(A_n)_{n\in\N}$ satisfies the assumptions of \Cref{pro:ConseqFurstenberg}.
\end{lemma}

\begin{proof}
First, observe that for all column matrices $v\in\R_{\ge 0}^{(M+N)p}$ with a positive first entry, for all $\ell\in\{0,\ldots,(M+N)p-1\}$ and all $n_1,\ldots,n_\ell\in\Z$, the $\ell+1$ first entries of the column matrix $A_{n_1}\cdots A_{n_\ell}v$ are positive. This can be easily shown by induction on the number $\ell$ of matrices in the product.

Second, it is easily seen that for all non-zero column matrices $v\in\R_{\ge 0}^{(M+N)p}$ and for all $n_1,\ldots,n_{Mp}\in\Z$, at least one of the last $Np$ entries of the column matrix $A_{n_1}\cdots A_{n_{Mp}}v$ is positive. 

Third, let us show that for all $n\in\Z$ and all $j\in\{Mp+1,\ldots,(M+N)p\}$, the first entry of the column matrix $A_{n+Np-1}\cdots A_n\mathbf{e}_j$ is positive. We have the equalities
\[
 \renewcommand{\arraystretch}{1.3}
   \begin{NiceArray}{@{} R<{{}} @{} LL}
    A_n\mathbf{e}_j                         &= a_{j-n,j}\mathbf{e}_1+\mathbf{e}_{j+1} 
                                            & \Block[c]{3-1}{
                                                \begin{tabular}{@{}l@{}}
                                                \text{empty if}\\
                                                \text{$j=(M{+}N)p$}
                                                \end{tabular}} \\
    A_{n+1}\mathbf{e}_{j+1}                 &= a_{j-n,j+1}\mathbf{e}_1+\mathbf{e}_{j+2} \\
                                            &\; \vdots \\
    A_{n+(M+N)p-j-1}\mathbf{e}_{(M+N)p-1}   &= a_{j-n,(M+N)p-1}\mathbf{e}_1+\mathbf{e}_{(M+N)p} \\
    A_{n+(M+N)p-j}\mathbf{e}_{(M+N)p}       &= a_{j-n,(M+N)p}\mathbf{e}_1+\mathbf{e}_{Mp+1} \\
    A_{n+(M+N)p-j+1}\mathbf{e}_{Mp+1}       &= a_{j-n-Np,Mp+1}\mathbf{e}_1+\mathbf{e}_{Mp+2} 
                                            & \Block[c]{3-1}{\begin{tabular}{@{}l@{}}
                                                \text{empty if}\\
                                                \text{$j=Mp+1$}
                                                \end{tabular}} \\
    A_{n+(M+N)p-j+2}\mathbf{e}_{Mp+2}       &= a_{j-n-Np,Mp+2}\mathbf{e}_1+\mathbf{e}_{Mp+3} \\
                                            &\; \vdots \\
    A_{n+Np-1}\mathbf{e}_{j-1}              &= a_{j-n-Np,j-1}\mathbf{e}_1+\mathbf{e}_j 
    \CodeAfter 
    \SubMatrix.{1-1}{4-2}\} 
    \SubMatrix.{6-1}{9-2}\}
    \end{NiceArray}
\]
Since at least one of the $Np$ digits $a_{j-n,Mp+1},\ldots,a_{j-n,(M+N)p}$ is nonzero, we see that the first entry of the row matrix $A_{n+Np-1}\cdots A_n\mathbf{e}_j$ is positive.

Putting the previous two observations together, we find that for all $n\in\Z$ and all non-zero column matrices $v\in\R_{\ge 0}^{(M+N)p}$, the first entry of the column matrix $A_{n+(M+N)p-1}\cdots A_nv$ is positive. The first observation then gives us that the column matrix $A_{n+2(M+N)p-1}\cdots A_nv$ is positive. Since $v$ is arbitrary, the matrix $A_{n+2(M+N)p-1}\cdots A_n$ must be positive. Since the sequence of matrices $(A_n)_{n\in\Z}$ is periodic (with period $p$), and since $(A_n)_{1,1}$ is clearly positive for all $n$, we obtain the conclusion.
\end{proof}

\begin{proposition}
\label{pro:ParryBase}
Let $p\ge 1$ and let $\mathbf{a}_0,\ldots,\mathbf{a}_{p-1}$ be ultimately periodic sequences of non-negative integer digits, not starting in $0$ and not ending in $0^\omega$. Then there exists an alternate base $B$ of length $p$ such that $\val_{S^i(B)}(0\bigcdot \mathbf{a}_i)=1$ for all $i\in\{0,\ldots,p-1\}$.
\end{proposition}

\begin{proof}
Without loss of generality, we assume that $\mathbf{a}_0,\ldots,\mathbf{a}_{p-1}$ all have preperiod $Mp$ and period $Np$ for some $M,N\in\N$. Let $(A_n)_{n\in\Z}$ be the sequence of matrices from \Cref{lem:ExistenceParryFurstenbergOK}. Let thus $(\gamma_n)_{n\in\Z}$ and $(f_n)_{n\in\Z}$ be the sequences given by \Cref{pro:ConseqFurstenberg}, which are periodic with period $p$. For all $n\in\Z$, set $\beta_n=\gamma_{-n}$ and let $\delta=\beta_0\cdots\beta_{p-1}$. By \Cref{pro:EventPeriodic} with $k=(M+N)p$ and $h=Mp$, and keeping in mind the periodicity of $f$, we obtain that 
\begin{itemize}
    \item $\beta_n>1$
    \item $1=\displaystyle{\sum_{j=1}^{Mp}} \frac{a_{n,j}}{\beta_{n-1}\cdots\beta_{n-j}}
            +\frac{f_{-n,Mp+1}}{\beta_{n-1}\cdots\beta_{n-Mp}}$
    \item $f_{-n,Mp+1}=\displaystyle{\sum_{j=Mp+1}^{(M+N)p}} \frac{a_{n,j}}{\beta_{n-Mp-1}\cdots\beta_{n-j}}
            +\frac{f_{-n,Mp+1}}{\delta^N}$
\end{itemize}
for all $n\in\Z$. The third item can be rewritten
\[
    f_{-n,Mp+1}=\frac{\delta^N}{\delta^N-1}\sum_{j=Mp+1}^{(M+N)p} \frac{a_{n,j}}{\beta_{n-Mp-1}\cdots\beta_{n-j}}.
\]
Substituting this in the second item yields
\[
    1   = \sum_{j=1}^{Mp} \frac{a_{n,j}}{\beta_{n-1}\cdots\beta_{n-j}}
            +\frac{\delta^N}{\delta^N-1}\sum_{j=Mp+1}^{(M+N)p} \frac{a_{n,j}}{\beta_{n-1}\cdots\beta_{n-j}} 
        = \val_{S^n(B)}(\mathbf{a}_n)
\]
where we have set $B=(\beta_n)_{n\in\Z}$.
\end{proof}

The following proposition provides lower and upper bounds on the alternate base found in \Cref{pro:ParryBase}.

\begin{proposition}
\label{prop:bounds} 
Let $p\ge 1$ and for each $i\in\{0,\ldots,p-1\}$, let $\mathbf{a}_i=a_{i,1}a_{i,2}\cdots$ be an ultimately periodic sequence of non-negative integer digits, not starting in $0$ and not ending in $0^\omega$. Let $C=Hp+1$ where 
\[
    H=\max\{a_{i,n} : i\in\{0,\ldots,p-1\},\ n\ge 1\}
\]
and let
\begin{multline*}
    L=\min\{\ell\ge 2 : \forall i\in\{0,\ldots,p-1\},
    \text{ the prefix } a_{i,1}\cdots a_{i,\ell} \text{ of } \mathbf{a}_i \\ 
    \text{ contains at least two non-zero digits}\}.
\end{multline*}
Then the alternate base $B=(\beta_{p-1},\ldots,\beta_0)$ found in \Cref{pro:ParryBase} is such that 
\[
    \frac{C^L}{C^L-1}<\beta_i\le C
\]
for all $i\in\{0,\ldots,p-1\}$.
\end{proposition}

\begin{proof} 
Let $i\in\{0,\ldots,p-1\}$ and set $\delta=\beta_0\cdots\beta_{p-1}$. As usual, we consider indices reduced modulo $p$. We have
\begin{align*}
    1   &=\val_{S^{i+1}(B)}(0\bigcdot \mathbf{a}_{i+1}) \\
        &\le \val_{S^{i+1}(B)}(0\bigcdot H^\omega) \\
        &=H\left(\sum_{j=0}^{p-1}\frac{1}{\beta_i\cdots\beta_{i-j}}\right)\frac{\delta}{\delta-1} \\
        &\le H\cdot \frac{p}{\beta_i}\cdot \frac{\delta}{\delta-1} \\
        &\le \frac{Hp}{\beta_i-1}.
\end{align*}
We get $\beta_i\le C$. We then have
\[
    1=\val_{S^{i+1}(B)}(0\bigcdot \mathbf{a}_{i+1})>\frac{1}{\beta_i}+\frac{1}{\beta_i\cdots\beta_{i-L+1}}
    \ge \frac{1}{\beta_i}+\frac{1}{C^L}.
\]
Therefore, we obtain $\beta_i> \frac{C^L}{C^L-1}$.
\end{proof}

In the next result, given $p$ sequences of digits, we again obtain an alternate base for which these sequences evaluate to $1$. Compared to~\Cref{pro:ParryBase}, we no longer require each given sequence to be ultimately periodic. The proof uses a reduction to the case of~\Cref{pro:ParryBase}. 

\begin{proposition}
\label{pro:existenceAlternateNonParry} 
Let $p\ge 1$ and for each $i\in\{0,\ldots,p-1\}$, let $\mathbf{a}_i$ be a sequence over a finite alphabet of non-negative integer digits that is lexicographically greater than $10^\omega$. Then there exists an alternate base $B$ of length $p$ such that $\val_{S^i(B)}(0\bigcdot \mathbf{a}_i)=1$ for all $i\in\{0,\ldots,p-1\}$.
\end{proposition}

\begin{proof}
Write $\mathbf{a}_i=a_{i,1}a_{i,2}\cdots$ for every $i\in\{0,\ldots,p-1\}$. We consider the `quasi-greedy version' of the sequences $\mathbf{a}_0,\ldots,\mathbf{a}_{p-1}$, which we denote by $\mathbf{b}_0,\ldots,\mathbf{b}_{p-1}$. That is, we define  
\[
\mathbf{b}_i = 
\begin{cases}
    a_{i,1}\cdots a_{i,\ell-1}(a_{i,\ell}-1)\mathbf{b}_{i-\ell}, & \text{ if } \mathbf{a}_i=a_{i,1}\cdots a_{i,\ell}0^\omega \text{ with }a_{i,\ell}\ge 1,\\
    \mathbf{a}_n, & \text{ otherwise}
\end{cases}
\]
where the first indices are considered modulo $p$. For each $i\in\{0,\ldots,p-1\}$, the sequence $\mathbf{b}_i=b_{i,1}b_{i,2}\cdots$ does not start with a zero digit and does not end in $0^\omega$. Let $L$ be the least integer such that for all $i\in\{0,\ldots,p-1\}$, the prefix $b_{i,1}\cdots b_{i,L}$ has at least two non-zero digits. For $i\in\{0,\ldots,p-1\}$ and $N\ge L$, we define
\[
    \mathbf{b}_i^{(N)} = b_{i,1} \cdots b_{i,N}\,1^\omega.
\]
By \Cref{pro:ParryBase}, there exists an alternate base $B^{(N)}=(\beta^{(N)}_{p-1},\ldots,\beta^{(N)}_0)$ such that $\val_{S^i(B^{(N)})}(0\bigcdot \mathbf{b}_i^{(N)})=1$ for all $i\in\{0,\ldots,p-1\}$. By \Cref{prop:bounds}, there exist uniform bounds $c$ and $C$, with $1<c<C$, such that all $\beta^{(N)}_i$ belong to the interval $(c,C]$. Then we have
\[
    1= \sum_{j=1}^N\frac{b_{i,j}}{\beta^{(N)}_{i-1}\cdots\beta^{(N)}_{i-j}} 
        +\sum_{j=N+1}^{+\infty}\frac{1}{\beta^{(N)}_{i-1}\cdots\beta^{(N)}_{i-j}}
    =  \sum_{j=1}^{+\infty}\frac{b_{i,j}}{\beta^{(N)}_{i-1}\cdots\beta^{(N)}_{i-j}} + \mathrm{E}(i,N)
\] 
with 
\[
    \mathrm{E}(i,N)=\sum_{j=N+1}^{+\infty}\frac{1-b_{i,j}}{\beta^{(N)}_{i-1}\cdots\beta^{(N)}_{i-j}}.
\]
Setting $H=\max\{b_{i,n} : i\in\{0,\ldots,p-1\},\ n\ge 1\}$, we obtain that 
\[
    |\mathrm{E}(i,N)|\le \sum_{j=N+1}^{+\infty}\frac{H}{c^j}=\frac{H}{c^N(c-1)}.
\]
As the $p$-tuples $B^{(N)}$ all belong to the compact set $[c,C]^p$, the sequence $(B^{(N)})_{N\ge L}$ has an accumulation point in $[c,C]^p$, which we denote by $B=(\beta_{p-1},\ldots,\beta_0)$. Letting $N$ grow while following of subsequence of indices such that $B^{(N)}$ tends to $B$, we obtain that $\val_{S^i(B)}(0\bigcdot \mathbf{b}_i)=1$ for all $i\in\{0,\ldots,p-1\}$. This implies that the original sequences $\mathbf{a}_i$ have $S^i(B)$-value equal to $1$ as well, i.e., that $\val_{S^i(B)}(\mathbf{a}_i)=1$ for all $i\in\{0,\ldots,p-1\}$.
\end{proof}

For an alternate base $B$ of length $p$, the Parry conditions given in \Cref{cor:CantorParryCondition} become $t_{i,j+1}t_{i,j+2}\cdots <_{\lex} \mathbf{d}_{i-j}$ and $d_{i,j+1}d_{i,j+2}\cdots \le_{\lex} \mathbf{d}_{i-j}$ for all $i\in\{0,\ldots,p-1\}$ and $j\ge 1$, where the index in $\mathbf{d}_{i-j}$ is seen modulo $p$. Our main result is the following one. It shows that given a list of $p$ sequences satisfying the Parry condition, there exists a unique alternate base with these sequences precisely being the corresponding greedy or quasi-greedy expansions of $1$.

\begin{theorem}
\label{thm:existenceAlternateNonParry}  
Let $p$ be a fixed positive integer and for all $i\in\{0,\ldots,p-1\}$, let $\mathbf{a}_i=a_{i,1}a_{i,2}\cdots$ be a sequence of non-negative integers lexicographically greater than $10^\omega$.
Assume moreover that for all $i\in\{0,\ldots,p-1\}$, if $\mathbf{a}_i$ ends in $0^\omega$ then
\begin{equation}
    \label{eq:Parry-condition-greedy}
    a_{i,j+1}a_{i,j+2}\cdots <_{\lex} \mathbf{a}_{i-j} \text{ for all }j\ge 1,
\end{equation}    
and if $\mathbf{a}_i$ does not end in $0^\omega$ then
\begin{equation}
    \label{eq:Parry-condition-quasi-greedy}
    a_{i,j+1}a_{i,j+2}\cdots \le_{\lex} \mathbf{a}_{i-j} \text{ for all }j\ge 1,
\end{equation}    
where the index in $\mathbf{a}_{i-j}$ is considered modulo $p$. 
Then there exists an alternate base $B$ of length $p$ such that $\mathbf{t}_i =\mathbf{a}_i $ for all $i \in \{0,\ldots, p-1\}$ such that $\mathbf{a}_i$ ends in $0^\omega$ and $\mathbf{d}_i =\mathbf{a}_i $ for all $i \in \{0,\ldots, p-1\}$ such that $\mathbf{a}_i$ does not end in $0^\omega$. Moreover, the following uniqueness properties hold.
\begin{enumerate}
    \item \label{eq:uniqueness1} If $\mathbf{a}_0,\ldots,\mathbf{a}_{p-1}$ are all ultimately periodic, then the base $B$ is unique. 
    \item \label{eq:uniqueness2} If $\mathbf{a}_0,\ldots,\mathbf{a}_{p-1}$ all start with a digit that is greater than $1$, then the base $B$ is unique. 
\end{enumerate}
\end{theorem}

\begin{proof}
The existence part of the result is a straightforward consequence of \Cref{pro:existenceAlternateNonParry,prop:uniqueness}. Now, let us consider the two uniqueness statements. First, suppose that $\mathbf{a}_0,\ldots,\mathbf{a}_{p-1}$ are all ultimately periodic. These $p$ sequences may or may not end in $0^\omega$. As in the proof of \Cref{pro:existenceAlternateNonParry}, we consider the quasi-greedy version $\mathbf{b}_0,\ldots,\mathbf{b}_{p-1}$ of the sequences $\mathbf{a}_0,\ldots,\mathbf{a}_{p-1}$. By \cite[Theorem 6.4]{Charlier&Cisternino&Masakova&Pelantova:2025}, at most one base $B$ may be such that $\mathbf{d}_i=\mathbf{b}_i$ for all $i\in\{0,\ldots,p-1\}$. Since we have assumed the lexicographic conditions \eqref{eq:Parry-condition-greedy} and \eqref{eq:Parry-condition-quasi-greedy}, for any such base $B$, we necessarily have $\mathbf{t}_i=\mathbf{a}_i$ for all $i\in\{0,\ldots,p-1\}$ such that $\mathbf{a}_i$ ends in $0^\omega$ and $\mathbf{d}_i =\mathbf{a}_i $ for all $i \in \{0,\ldots, p-1\}$ such that $\mathbf{a}_i$ does not end in $0^\omega$. This gives us Item \eqref{eq:uniqueness1}. The second uniqueness statement, i.e., Item \eqref{eq:uniqueness2}, follows from \Cref{cor:uniqueness2}.
\end{proof}

The following two corollaries are immediate.

\begin{corollary}
Let $p$ be a fixed positive integer and for all $i\in\{0,\ldots,p-1\}$, let $\mathbf{a}_i=a_{i,1}a_{i,2}\cdots$ be a sequence of non-negative integers ending in $0^\omega$, lexicographically greater than $10^\omega$, and satisfying $a_{i,j+1}a_{i,j+2}\cdots <_{\lex} \mathbf{a}_{i-j}$ for all $i\in\{0,\ldots,p-1\}$ and $j\ge 1$, where the index in $\mathbf{a}_{i-j}$ is considered modulo $p$. Then there exists a unique alternate base $B$ of length $p$ such that $\mathbf{t}_i =\mathbf{a}_i $ for all $i \in \{0,\ldots, p-1\}$. 
\end{corollary}

\begin{corollary}
Let $p$ be a fixed positive integer and for all $i\in\{0,\ldots,p-1\}$, let $\mathbf{a}_i=a_{i,1}a_{i,2}\cdots$ be a sequence of non-negative integers not starting in $0$, not ending in $0^\omega$, and satisfying $a_{i,j+1}a_{i,j+2}\cdots \le_{\lex} \mathbf{a}_{i-j}$ for all $i\in\{0,\ldots,p-1\}$ and $j\ge 1$, where the index in $\mathbf{a}_{i-j}$ is considered modulo $p$. Then there exists an alternate base $B$ of length $p$ such that $\mathbf{d}_i =\mathbf{a}_i $ for all $i \in \{0,\ldots, p-1\}$. 
Moreover, the base $B$ is unique if either $\mathbf{a}_0,\ldots,\mathbf{a}_{p-1}$ are all ultimately periodic, or if $\mathbf{a}_0,\ldots,\mathbf{a}_{p-1}$ all start with a digit that is greater than $1$.
\end{corollary}



\section{A family of sequences faithfully coding B-integers}
\label{sec:specific}

Let $B$ be a Cantor real base. A real number is called a \emph{$B$-integer} if it has no fractional part with respect to $B$, i.e., if its $B$-expansion is of the form $\langle x\rangle_B=a_{N-1}\cdots a_0\bigcdot 0^\omega$. These numbers form a discrete subset of $\R$, and thus we can encode the distances between consecutive elements of this set by an infinite sequence. If there is precisely one symbol encoding each possible distance, we call this encoding the \emph{faithful} coding of the $B$-integers (which is clearly unique, up to renaming the letters). It was shown in \cite{Charlier&Cisternino&Masakova&Pelantova:2025} that this sequence has an $S$-adic representation. Let us recall the necessary definitions in order to state this result in a rigorous form. 

\begin{definition}
\label{Def:B-integers}
For all $m,n\in \N$, we define
\[
	\Delta_{m,n} = \val_{S^m(B)} (0\bigcdot d_{m+n,n+1}d_{m+n,n+2} \cdots).
\]
Then, for all $m\in\N$, we define a map $\pi_m\colon\N\to\N$ by
\[
    \pi_m(n)= 
    \begin{cases} 
    \pi_m(n'),  & \text{if }\Delta_{m,n} = \Delta_{m,n'} \text{ for some }n'<n;\\
    n,          & \text{otherwise}
    \end{cases}
\]
and an alphabet $A_m=\pi_m(\N)$. Finally, for all $m\in\N$, we define a substitution $\varphi_m\colon A_{m+1}^*\to A_m^*$ by
\[
	\varphi_m(n)=0^{d_{m+n+1,n+1}}\pi_m(n+1).
\]
\end{definition}

\begin{theorem}[{\cite[Corollary 5.5]{Charlier&Cisternino&Masakova&Pelantova:2025}}] 
\label{thm:S-adic-B-integers}
Let $B$ be a Cantor real base and $m\in\Z$. The $S$-adic sequence $\lim_{n\to \infty}\varphi_m \cdots \varphi_{m+n-1}(0)$ is the faithful coding of the $S^m(B)$-integers. 
\end{theorem}

Our aim in this section is to define a family of substitutions such that any $S$-adic sequence obtained from this family can be obtained as the faithful coding of the $B$-integers for some Cantor real base $B$. We start by defining the family of substitutions.

\begin{definition}
\label{def:eta_c}
Let $k\ge 2$ be a fixed integer. For a $k$-tuple $\mathbf{c}=(c_1,\ldots,c_k)$ of positive integers, we define a morphism $\eta_\mathbf{c}\colon\{0,\ldots,k-1\}^*\to \{0,\ldots,k-1\}^*$ by
\[
    \eta_\mathbf{c}(j)
    = \begin{cases} 
    0^{c_{j+1}} (j +1),  & \text{if } j<k-1; \\
    0^{c_k},     & \text{if } j=k-1.
    \end{cases}    
\]
Then we define
\begin{equation}
\label{eq:Substitutions}
    S=\{\eta_\mathbf{c} : \mathbf{c}=(c_1,\ldots,c_k),\ c_1\ge \cdots\ge c_k\ge 1\}. 
\end{equation}     
\end{definition}

\begin{theorem}
\label{thm:k-aryWords} 
Any $S$-adic sequence of the form $\lim_{n\to\infty}\psi_0\psi_1\cdots \psi_{n-1}(0)$, where $\psi_n\in S$ for all $n\ge 0$, is the faithful coding of the $B$-integers for some Cantor real base $B$.
\end{theorem}

\begin{proof} 
For all $n\ge 0$, let $\psi_n\in S$ and let $(a_{n+1,1},\ldots,{a_{n+1,k}})$ be the corresponding $k$-tuple of parameters. For $n\le 0$, we set $(a_{n,1},\ldots,{a_{n,k}})=(1,\ldots,1)$. For all $n\in\Z$, we define a square matrix of size $k$ by
\[
    A_n=\begin{pNiceArray}{ccc|c}
  a_{-n+1,1} & \cdots & a_{-n+1,k-1} & a_{-n+1,k} \\
  \hline
  \Block{3-3}<\Large>{\mathbf{I}} & \\
  &&& \mathbf{0} \\
  &&&
\end{pNiceArray}
\]
where $\mathbf{I}$ is the identity matrix of size $k-1$ and $\mathbf{0}$ is the zero matrix column of size $k-1$. 
The product of any $k$ such matrices is positive. Moreover, the product $A_{n+k}\cdots A_{n+1}$ is equal to $A_0^k$ for all $n\ge 0$. Therefore, the sequence $(A_n)_{n\in\N}$ satisfies the hypotheses of \Cref{pro:ConseqFurstenberg} and we may apply \Cref{pro:ConseqFurstenberg,pro:FiniteExpansion}.

Let $(\gamma_n)_{n\in\Z}$ be the sequence given by \Cref{pro:ConseqFurstenberg} and for all $n\in\Z$, set $\beta_n=\gamma_{-n}$. By \Cref{pro:FiniteExpansion}, we obtain that $\beta_n>1$ and 
\[
    1=\sum_{j=1}^k \frac{a_{n-j+1,j}}{\beta_{n-1}\cdots\beta_{n-j}}
\]
for all $n\in\Z$. By setting $B=(\beta_n)_{n\in\Z}$, this can be reexpressed as 
\[
    1=\val_{S^n(B)}(0\bigcdot a_{n,1}a_{n-1,2}\cdots a_{n-k+1,k}0^\omega)
\]
for all $n\in\Z$. By \Cref{cor:greedyExpansion}, in order to obtain that the greedy expansions $\mathbf{t}_n$ of $1$ are given by
\[
    \mathbf{t}_n=a_{n,1}a_{n-1,2}\cdots a_{n-k+1,k}0^\omega
\]
it suffices to check that these sequences satisfy the Parry condition, i.e., that 
\[
    a_{n-j,j+1}a_{n-j-1,j+2}\cdots a_{n-k+1,k}0^\omega
    <_{\lex} a_{n-j,1}a_{n-j-1,2}\cdots a_{n-j-k+1,k}0^\omega
\]
for all $n\in\Z$ and all $j\ge 1$. These inequalities are indeed satisfied since we have assumed that
\begin{equation}
\label{eq:inequalities}
    a_{n,1}\ge a_{n,2}\ge \cdots \ge a_{n,k}\ge 1.
\end{equation}
The quasi-greedy expansions of $1$ are then given by
\[
    \mathbf{d}_n=a_{n,1}a_{n-1,2} \cdots a_{n-k+2,k-1} (a_{n-k+1,k}-1) \mathbf{d}_{n-k}.
\]
From this, we obtain 
\begin{align*}
    \Delta_{n,0} 
        &= \Delta_{n,\ell k} 
        = \val_{S^n(B)}(0\bigcdot a_{n,1}a_{n-1,2}\cdots a_{n-k+1,k}0^\omega)
        = 1, \\ 
    \Delta_{n,1} 
        &= \Delta_{n,\ell k+1} 
        = \val_{S^n(B)}(0\bigcdot a_{n,2}a_{n-1,3}\cdots a_{n-k+2,k}0^\omega), \\ 
        &\,\ \vdots \\
    \Delta_{n,k-1} 
        &= \Delta_{n,\ell k+k-1} 
        = \val_{S^n(B)}(0\bigcdot a_{n,k}0^\omega)
\end{align*}
for all $\ell\in\N$. Due to the inequalities~\eqref{eq:inequalities}, we derive that 
\[
    \Delta_{n,0} > \Delta_{n,1} > \cdots > \Delta_{n,k-1}.
\]
With the notation of \Cref{Def:B-integers}, we have $\pi_n(\N)=\{0,\ldots,k-1\}$ for all $n\in\Z$. (In the context of \cite{Charlier&Cisternino:2021}, this means that the distances between consecutive $S^n(B)$-integers take precisely $k$ values.) 

By \Cref{thm:S-adic-B-integers}, in order to finish the proof, it suffices to show that $\varphi_n=\psi_n$ for all $n\in\Z$. By \Cref{Def:B-integers}, for all $n\in\N$ and $j\in\{0,\ldots,k-1\}$, we have $\varphi_n(j)=0^{d_{n+j+1,j+1}}\pi_n(j+1)$. From what precedes, we know that 
\[
    d_{n+j+1,j+1}
    = \begin{cases} 
        a_{n+1,j+1},    & \text{if } j<k-1; \\
        a_{n+1,k}-1,          & \text{if } j=k-1
    \end{cases}    
\]
and that $\pi_n(j+1)=j+1$ if $j<k-1$ and $\pi_n(k)=0$. Therefore, for all $j\in\{0,\ldots,k-1\}$, we obtain that
\[
    \varphi_n(j)
    = \begin{cases} 
    0^{a_{n+1,j+1}} (j+1),  & \text{if } j<k-1; \\
    0^{a_{n+1,k}},          & \text{if } j=k-1,
    \end{cases}    
\]
which is precisely $\psi_n(j)$.
\end{proof}

\subsection{Arnoux-Rauzy sequences} 
\label{subsec:AR}

Arnoux-Rauzy sequences can be viewed as a generalization of Sturmian sequences to multiliteral alphabets. Consider the alphabet $\{0,\ldots,k-1\}$, and let $L_0,\ldots,L_{k-1}$ be the morphisms defined as follows: 
\[
    L_i \colon \{0,1,\dots, k-1\}^*\to \{0,1,\dots, k-1\}^*,\ 
    j\mapsto \begin{cases}
        i, & \text{if } j=i; \\
        ij, &\text{if } j\neq i
        \end{cases}
\]
for $i\in\{0,\ldots,k-1\}$.
It is a known fact that every standard  Arnoux-Rauzy sequence~$\mathbf{w}$ over the alphabet $\{0,1,\dots, k-1\}$ has an $S$-adic expansion using the $k$ substitutions $L_0, L_1,\dots,L_{k-1}$ \cite{Glen&Justin:2009,Fogg:2002}. More precisely, every standard Arnoux-Rauzy sequence can be written as 
\[
    \mathbf{w} = \lim_{n\to \infty} \psi_0\psi_1\cdots \psi_{n-1}(0)
\]
where $\psi_n\in \{L_0, L_1,\dots,L_{k-1}\}$ for all $n\in\N$.
In accordance with~\cite{Peltomaki:2024}, we say that a standard Arnoux-Rauzy sequence is \emph{regular} if its $S$-adic expansion can be expressed in the form  
\[
    \psi_0\psi_1\psi_2 \cdots = L_{ 
    i_1}^{a_1}L_{i_2}^{a_2}L_{i_3}^{a_{3}}\cdots,
\]
for positive integers $a_1,a_2,a_3\ldots$ and where the sequence $(i_n)_{n\geq 1}$ is purely periodic with period $k$ and contains every letter in $\{0,\ldots,k-1\}$. Note that every standard Sturmian sequence (i.e., a standard Arnoux-Rauzy sequence over a binary alphabet) is regular. Let us realize that up to renaming the letters of the alphabet,  we can assume that the $S$-adic representation of a regular Arnoux-Rauzy sequence is of the form 
\begin{equation}
\label{eq:Regular}
    \psi_0\psi_1\psi_2 \cdots = L_0^{a_1}L_1^{a_2}L_2^{a_3} \cdots L_{k-1}^{a_k}L_0^{a_{k+1}}L_1^{a_{k+2}}L_2^{a_{k+3}}\cdots
\end{equation}

\begin{example}
The most prominent ternary Arnoux-Rauzy sequence is the so-called Tribonacci sequence $\mathbf{t}$, the fixed point of the substitution $\varphi\colon 0\mapsto 01,\ 1\mapsto 02,\ 2\mapsto 0$. On the ternary alphabet $\{0,1,2\}$, we have 
\[
	L_0\colon\begin{cases}	
	0\mapsto 0\\
	1\mapsto 01\\
    2\mapsto 02
	\end{cases}\quad 
    L_1\colon\begin{cases}	
	0\mapsto 10\\
	1\mapsto 1\\
    2\mapsto 12
	\end{cases}\quad 
    L_2\colon\begin{cases}	
	0\mapsto 20\\
	1\mapsto 21\\
    2\mapsto 2.
	\end{cases}
\]
As the Tribonacci sequence $\mathbf{t}$ is fixed by $\varphi$, it is fixed by $\varphi^3$ as well. One can easily check that $\varphi^3 = L_0L_1L_2$. Consequently, $\mathbf{t}$ is a regular Arnoux-Rauzy sequence. 

The Tribonacci sequence is an example of a so-called \emph{$\beta$-substitution} as introduced by Fabre \cite{Fabre:1995}. Let us define this family of substitutions in the case of \emph{simple Parry numbers}, i.e., real numbers $\beta>1$ such that the expansion of $1$ in the R\'enyi numeration system with the base $\beta$ has the form $d_\beta(1) = t_1t_2\dots t_k0^\omega$. In this case, the $\beta$-substitution is given by 
\[
    0\mapsto 0^{t_1}1,\ 
    1\mapsto 0^{t_2}2,\ 
    \ldots,\
    k-2 \mapsto 0^{t_{k-1}}(k-1),\ 
    k-1 \mapsto 0^{t_k}.
\]
As shown in \cite{Bernat&Masakova&Pelantova:2007}, the fixed point of such a substitution is an Arnoux-Rauzy sequence if and only if $t_1=t_2=\cdots = t_{k-1}=:t$  and $t_k=1$. The fixed point has the $S$-adic expansion
$(L_0^t L_1^t \cdots L_{k-1}^t)^\omega$, hence it is a regular Arnoux-Rauzy sequence.      
\end{example}

We obtain as a consequence of \Cref{thm:k-aryWords} that all Arnoux-Rauzy sequences in this family can be expressed as the coding of the $B$-integers of some Cantor base.

\begin{corollary} 
\label{cor:AR} 
Every regular $k$-ary Arnoux-Rauzy sequence is the faithful coding of the $B$-integers for some Cantor real base $B$.  
In particular, it is the case of every standard Sturmian sequence.
\end{corollary}

\begin{proof} 
Let $R\colon \{0,\dots,k-1\}^*\to \{0,\dots,k-1\}^*$ be the morphism given by $j \mapsto (j+1) \bmod k$. It is readily seen that for any $i\in\{0,\dots,k-1\}$, one has $L_i=R^iL_0R^{-i}$ and therefore $L_i^c=R^iL_0^cR^{-i}$ for all $c\in\N$. Hence, the $S$-adic representation \eqref{eq:Regular} of any regular Arnoux-Rauzy sequence can be expressed as $(L_0^{a_1}R)(L_0^{a_2}R)(L_0^{a_3}R)\cdots$. If for $a\ge 1$, we define $\mathbf{c}=(a,a,\ldots,a,1)\in\N^k$, then the substitution $\eta_\mathbf{c}$ from \Cref{def:eta_c} satisfies $L_0^{a}R = \eta_\mathbf{c}$. By \Cref{thm:k-aryWords}, there exists a Cantor real base $B$ such that $\mathbf{w}$ is the faithful coding of the $B$-integers.
\end{proof}

\subsection{Sequences associated with N-continued fractions }
\label{subsec:NCF}

In  \cite{Langeveld&Rossi&Thuswaldner:2023}, Lange\-veld, Rossi and Thuswaldner study $N$-continued fractions of numbers, that is, expansions of the form 
\[
    x = \frac{N}{d_1+\frac{N}{d_2+ \cdots}}.
\]
Note that $1$-continued fractions coincide with the classical continued fractions. The authors introduce a family of substitutions over a binary alphabet: for $d\ge N$, let
\[
    \hat{\sigma}_d\colon 0\mapsto 0^d1,\ 1\mapsto 0^N,
\]  
see \cite[Definition 3.1]{Langeveld&Rossi&Thuswaldner:2023}.
To $x\in(0,1)\setminus\mathbb{Q}$ with the $N$-continued fraction $x=[d_1,d_2,d_3,\ldots]_N$ they assign the so-called {\it dual NCF sequence}  
\[
    \hat{\omega}(x,N)
    =\lim_{n\to\infty}\hat{\sigma}_{d_1}\hat{\sigma}_{d_2}\cdots \hat{\sigma}_{d_n}(0). 
\]
One can easily observe that the substitutions $\hat{\sigma}_{d_n}$ belong to the set of substitutions $S$ defined in \Cref{def:eta_c} for $k=2$. Therefore, \Cref{thm:k-aryWords} has the following corollary.   

\begin{corollary}
\label{cor:dual} 
Every dual NCF sequence $\hat{\omega}(x, N )$ is the faithful coding of the $B$-integers for some Cantor real base $B$.
\end{corollary}

\section{Comments}

\begin{enumerate}
 
\item In the case of non-alternate bases, both the existence and uniqueness proofs are missing. For some special cases of strings $\mathbf{a}_n$, $n\in\Z$, the existence of a base can be deduced in a similar way to the proof of \Cref{thm:k-aryWords}.   

\item The case of standard Sturmian sequences of \Cref{cor:AR} was already proved in \cite{Charlier&Cisternino&Masakova&Pelantova:2025}. In the same paper, even some non-standard Sturmian sequences were presented as coding of $B$-integers. It would be interesting to describe how diverse the set of Sturmian  words is hidden among the words encoding $B$-integers. 

\item Infinite sequences associated with the R\'enyi numeration system (i.e.,  for $p=1$ in our setting) have been and are being studied from various perspectives: closedness under mirror image, palindromic richness \cite{Ambroz&Masakova&Pelantova&Frougny:2006}, affine factor complexity \cite{Bernat&Masakova&Pelantova:2007}, critical exponent \cite{Balkova&Klouda&Pelantova:2011}, and recently attractors of these sequences \cite{Gheeraertetal2024}, \cite{Dvorakova&Moravcova:2025}. The listed properties are guaranteed in the R\'enyi systems with bases $\beta$ with a very special expansion of $1$. We believe that within the Cantor real base numeration system we will find wider classes of words with the mentioned properties.
\end{enumerate}

\section*{Acknowledgments}
Savinien Kreczman is supported by the F.R.S.-FNRS Research Fellow grant 1.A.789.23F, by a Fédération Wallonie-Bruxelles Travel Grant and by a WBI Excellence WORLD grant. Avec le soutien de Wallonie-Bruxelles International. 
\includegraphics[scale=0.01]{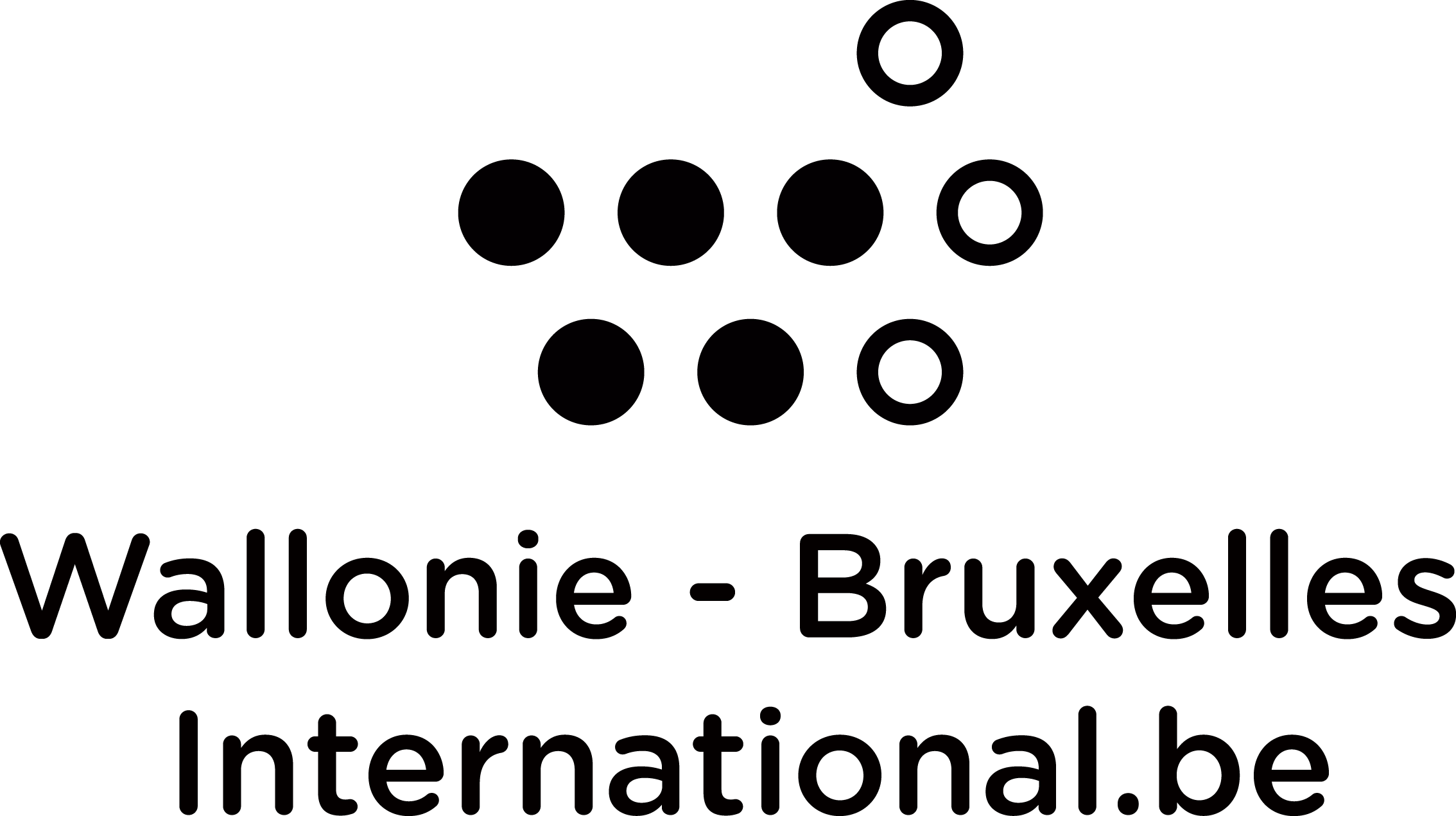}


\bibliographystyle{abbrv}
\bibliography{Biblio_Existence}

\end{document}